\documentclass{amsart}
\usepackage{amssymb}
\usepackage{epsfig}
\usepackage{amsthm}
\usepackage{amsfonts,amsmath,bbm,mathrsfs,esint}
\usepackage{subfig}
\usepackage{graphicx}
\usepackage{psfrag}
\usepackage[latin1]{inputenc}
\usepackage{hyperref}

\newcommand{\N}{\mathbb{N}}

\newcommand{\R}{\mathbb{R}}
\newcommand{\C}{\mathbb{C}}

\newtheorem{theorem}{Theorem}[section]
\newtheorem{lemma}[theorem]{Lemma}
\newtheorem{corollary}[theorem]{Corollary}
\newtheorem{remark}[theorem]{Remark}

\def\d{\,\mathrm d}		% roman d as in dx
\def\cont{\mathscr C}		% smooth functions (as in C^2)
\def\bg{{\mathbbm 1}}		% background, \Lambda_\bg or u_\bg
\def\la{\langle}
\def\ra{\rangle}

\begin{document}

\title[Point measurements and the Calder\'on problem in the plane]
{Point measurements for a Neumann-to-Dirichlet map 
and the Calder\'on problem in the plane}

\author{Nuutti Hyv\"onen}
\address{Aalto University, Department of Mathematics and Systems Analysis, FI-00076 Aalto, Finland}
\email{nuutti.hyvonen@aalto.fi}
\urladdr{http://users.tkk.fi/nhyvonen/}

\author{Petteri Piiroinen}
\address{University of Helsinki, P.O. Box 68, FI-00014 Helsinki
 Finland} 
\email{petteri.piiroinen@helsinki.fi}

\author{Otto Seiskari}
\address{Aalto University, Department of Mathematics and Systems Analysis, FI-00076 Aalto, Finland}
\email{otto.seiskari@aalto.fi}

\thanks{This work was supported by the
Academy of Finland (the Centre of Excellence in Inverse Problems Research and project 135979)}

\subjclass[2010]{35R30, 35Q60}

\keywords{Calder\'on problem, Neumann-to-Dirichlet map, point measurements,
(bi)sweep data, partial data}

\hypersetup{%
	pdftitle={\shorttitle},%
	pdfauthor={},%
	pdfsubject={},%
	pdfkeywords={},%
	pdfborder={0 0 0},%
	colorlinks={false}%
}

\begin{abstract}
This work considers properties of the Neumann-to-Dirichlet map
for the conductivity equation under the assumption that the conductivity is
identically one close to the boundary of the examined smooth, bounded
and simply connected domain. It is demonstrated that the so-called
bisweep data, i.e., the (relative) potential differences between two boundary 
points when delta currents of opposite signs are applied at the very same
points, uniquely determine the whole Neumann-to-Dirichlet map. In
two dimensions, the bisweep data extend as a holomorphic function of two
variables to some (interior) neighborhood of the product boundary. It follows
that the whole Neumann-to-Dirichlet map is characterized by the
derivatives of the bisweep data at an arbitrary point.
On the diagonal of the product boundary, these derivatives can be 
given with the help of the derivatives of the (relative) boundary potentials 
at some fixed point caused by the distributional current densities supported 
at the  same point, 
and thus such point measurements uniquely define the Neumann-to-Dirichlet map. 
This observation also leads to a 
new, truly local uniqueness result for the so-called Calder\'on inverse 
conductivity problem.
\end{abstract}

\maketitle

\section{Introduction.}
\label{sec:intro}
In this work, we consider properties of the Neumann-to-Dirichlet operator, 
i.e., the current-to-voltage boundary map for the conductivity equation
\begin{equation}
\label{basic_eq}
\nabla \cdot(\sigma\nabla u) \,=\, 0 \quad \mbox{in $D$}
\end{equation}
assuming the conductivity  $\sigma$ is identically one in some interior 
neighborhood of
the boundary of the smooth, simply connected and bounded domain 
$D \subset \R^n$, $n \geq 2$. In particular, we are interested in what
kind of `point measurements' uniquely characterize the Neumann-to-Dirichlet
map; as a by-product, we will obtain a new result
for the two-dimensional Calder\'on problem with partial data.

Our main analytic tool is the so-called bisweep data, 
which are the (relative) potential differences between two points on 
$\partial D$ when delta distribution current
is driven between the very same points (cf.~\cite{Hakula11}). 
We demonstrate that the bisweep data
uniquely determine the whole Neumann-to-Dirichlet map for any symmetric 
anisotropic $\sigma \in L^{\infty}(D, \R^{n\times n})$, $\sigma \geq c I > 0$, 
conductivity as long as it satisfies the isotropic 
homogeneity assumption near the object boundary. 

The completeness of the bisweep data has remarkable consequences in the
two-dimensional case, when $D$ can be identified with a part of the 
complex plane. By extending an argument in 
\cite{Hakula11,Hanke11b,Hanke11a,Hyvonen11}, the bisweep data can be 
continued as a complex analytic function of two variables to some 
(interior) neighborhood of $\partial D \times \partial D \subset \C^2$. 
In particular,
all bisweep data are determined by the corresponding derivatives at any
fixed point $(z_1,z_2)$ on $\partial D \times \partial D$. In case
$z_1 = z_2 = z$, these
derivatives can be presented with the help of the derivatives at 
$z$ of the relative boundary potentials caused by the current densities that are the 
derivatives of the delta distribution located at $z$. 
In other words, sampling the relative potentials originating from the 
distributional currents supported at a fixed $z \in \partial D$ by the very
same distributions determines the whole Neumann-to-Dirichlet map.
It also follows that the Neumann-to-Dirichlet map can be recovered from 
the bisweep data on  any countably infinite set 
$\Gamma \times \Gamma \subset \partial D
\times \partial D$ with an accumulation point at $(z,z)$.
(A related result that assumes continuum 
measurements but also provides a stability estimate can be found in 
\cite{Alessandrini12}.)
On the other hand, it is known that the Neumann-to-Dirichlet map 
uniquely defines an isotropic $L^\infty$-conductivity in two dimensions,
as demonstrated in~\cite{Astala06}.

The problem of determining an isotropic conductivity in 
\eqref{basic_eq} from
information on the Cauchy data of the corresponding solutions is called the 
Calder\'on  problem. It was proposed by Calder\'on in \cite{Calderon80} and 
tackled by many renowned mathematicians since. In dimensions $n\geq3$, the 
first global uniqueness result for $\mathscr{C}^2$-conductivities was proven in 
\cite{Sylvester87}, and extended for less regular conductivities in 
\cite{Brown03,Paivarinta03}.
In two dimensions the first global uniqueness result was provided by 
\cite{Nachman96} for $\mathscr{C}^2$-conductivities. Subsequently, 
the regularity assumptions were relaxed in \cite{Brown97} and,
in particular, \cite{Astala06} proved uniqueness for general isotropic 
$L^\infty$-conductivities. 

All the above mentioned articles assume that the Cauchy data are known on
all of $\partial D$, but there also exist several results considering the partial
data problem of having access only to some subset(s) of 
$\partial D$. To the best of the authors' knowledge, the most general
result for the partial data case in our two-dimensional setting for the
Calder\'on problem is currently found in \cite{Imanuvilov10}, where it is shown that 
an isotropic conductivity of smoothness $\mathscr{C}^{4, \alpha}$, 
$\alpha > 0$, is uniquely defined by the Dirichlet-to-Neumann map restricted to 
any open nonempty subset of $\partial D$. Compared to this result, the amount 
of data 
needed for our uniqueness theorem is considerably smaller: all
derivatives of the (relative) boundary potentials at a single point, caused
by distributional current densities supported at the same 
point, suffice.
Moreover, our result allows $L^\infty$-conductivities, but only under the 
important and arguably rather restrictive assumption of 
homogeneity of 
$\sigma$ in some interior neighborhood of $\partial D$.
For other results on the Calder\'on problem with partial data, we refer to
\cite{Bukhgeim02,Gebauer08,Isakov07,Kenig07,Kohn84,Kohn85,Knudsen06} and the references therein.

There also exists a vast literature\label{question:aniso}
on the (non)unique solvability of the Calder\'on problem for anisotropic
conductivities; see, e.g., 
\cite{Astala05,Imanuvilov10,Lee89,Nachman96,Sun03,Sylvester90}. In two 
dimensions it has been shown that an anisotropic $L^{\infty}$-conductivity is 
determined by the Dirichlet-to-Neumann map
up to a natural obstruction, i.e., up to a pushforward by an 
$H^1$-diffeomorphism that fixes the object boundary \cite{Astala05}; 
see \cite{Sylvester90} for the original ideas behind
such results. For the case of partial data, the most general result
is arguably found in \cite{Imanuvilov10},
where it is shown that an anisotropic conductivity of the class 
$\mathscr{C}^{7, \alpha}$, $\alpha > 0$, is defined up to the natural 
obstruction by the Dirichlet-to-Neumann map restricted to 
any open nonempty subset of $\partial D$. The results presented in this
work demonstrate that under the assumption of isotropic homogeneity
close to the object boundary, an anisotropic $L^\infty$-conductivity is
defined up to the natural obstruction by the above described 
point measurements for the Neumann-to-Dirichlet map.

The main reason for choosing to work with the 
Neumann-to-Dirichlet
map instead of the Dirichlet-to-Neumann map is that in practical {\em electrical
impedance tomography}, which is the imaging modality corresponding to the
Calder\'on problem, the natural boundary condition on nonaccessible parts of
 the object boundary is the homogeneous Neumann, not the homogeneous Dirichlet condition \cite{Cheney99, Cheng89,Somersalo92}. In particular, the bisweep
data can be approximated by real-life measurements performed by two small 
movable electrodes \cite{Hakula11,Hanke11c,Hyvonen11}.

This text is organized as follows. In Section~\ref{sec:main} we introduce
our setting and formulate the main results. Subsequently, 
Section~\ref{sec:proof} provides the corresponding proofs: In 
Section~\ref{sec:factor} we introduce a useful factorization for the relative
Neumann-to-Dirichlet map, Sections~\ref{sec:extension} and \ref{sec:extension2}
prove the complex analytic extension property for the bisweep data, and 
finally the actual proofs are formulated in Section~\ref{sec:uniqueness}.

\section{The setting and main results}
\label{sec:main}

Let $D \subset \R^n$, $n \geq 2$, be a simply connected and bounded domain 
with a  $\cont^{\infty}$-boundary. Assume that the symmetric conductivity 
$\sigma \in L^{\infty}(D, \R^{n \times n})$ satisfies
$$
\sigma \geq cI \quad \mbox{for } c > 0 \qquad {\rm and} \qquad 
\Sigma := {\rm supp} (\sigma - I) \mbox{ is a compact subset of $D$},
$$
where $I \in \R^{n\times n}$ is the identity matrix and 
the first condition is to be understood in the sense of positive definiteness
almost everywhere.  Our main result for the Calder\'on problem will be 
formulated only for the two-dimensional case, but some 
interesting intermediate theorems are valid independently of the
spatial dimension.

Consider the Neumann boundary value problem
\begin{equation}
\label{forward}
\nabla \cdot(\sigma\nabla u) \,=\, 0 \quad \mbox{in $D$}, \qquad
   \frac{\partial u}{\partial \nu} \,=\, f \quad \mbox{on $\partial D$}
\end{equation}
for a current density $f$ in 
\begin{equation}
\label{meanfree}
H^s_{\diamond}(\partial D) \, = \, \{ g \in H^s(\partial D) \ : \ 
\langle g, 1 \rangle_{\partial D} = 0 \}, 
\end{equation}
with some $s \in \R$. Here and in what follows, $\nu$ denotes the exterior
unit normal field of the respective domain, and we note that the dual of 
$H^s_{\diamond}(\partial D)$ is realized by
\begin{equation}
\label{uptoconstant}
H^{-s}(\partial D) / \C := H^{-s}(\partial D) / {\rm span} \{ 1 \}, \qquad s\in \R.
\end{equation}
It follows from standard theory for elliptic
boundary value problems that \eqref{forward} has a unique solution $u_\sigma$
in $(H^{\min \{1, s+3/2 \} }(D)\cap H^1_{\rm loc}(D) ) / \C$ and that 
the corresponding Neumann-to-Dirichlet map
\begin{equation}
\label{neumann:dirichlet:Hs:bounded}
\Lambda_\sigma : f \mapsto u_\sigma|_{\partial D}, \quad H^s_{\diamond}(\partial D) \to H^{s+1}(\partial D) / \C
\end{equation}
is well defined and bounded for any $s \in \R$ (cf., e.g., \cite{Hanke11a,Lions72}). We denote by $\bg \in L^{\infty}(D)$ the homogeneous unit conductivity
and note that the relative Neumann-to-Dirichlet map
$$
\Lambda_\sigma - \Lambda_\bg: f \mapsto  (u_\sigma - u_\bg)|_{\partial D}, \quad
\mathscr{D}'_{\diamond}(\partial D) \to \mathscr{D}(\partial D)/\C
$$
is well defined  (and bounded as an operator from $H^s_\diamond(\partial D)$
to $H^{-s}(\partial D) / \C$ for any $s \in \R$). Here, the mean-free
distributions $\mathscr{D}'_{\diamond}(\partial D)$ and the quotient space
of smooth boundary potentials  $\mathscr{D}(\partial D)/\C$ are defined
in accordance with \eqref{meanfree} and~\eqref{uptoconstant}; we also
use similar notations, $\mathscr{D}'(\partial D)/\C$ and 
$\mathscr{D}_{\diamond}(\partial D)$, when the roles of distributions and 
smooth test functions are reversed. This regularity result can be deduced
from standard elliptic theory (cf., e.g., \cite{Hanke11a,Lions72}), 
and it also follows from $\Lambda_\sigma$ and $\Lambda_\bg$ being 
pseudodifferential
operators with the same symbol because  $\sigma$ and $\bg$ coincide in some
interior neighborhood of $\partial D$ \cite{Lee89}.

We define the bisweep data as the function 
\begin{equation}
\label{bisweep}
\varsigma_\sigma: (x,y) \mapsto 
\left\langle (\delta_{x} - \delta_{y}), (\Lambda_\sigma - \Lambda_\bg) (\delta_{x} - \delta_{y}) \right \rangle_{\partial D}, \quad \partial D \times \partial D \to \R,
\end{equation}
where $\delta_z$ denotes the delta distribution at $z$ on $\partial D$.
This is a generalization of the concept of (standard) 
sweep data from \cite{Hakula11}, or the other way around, sweep data are the
restriction of bisweep data onto $\partial D \times \{y_0 \}$ for some 
fixed $y_0 \in \partial D$. What is more, the bisweep
data can be approximated by two-electrode measurements in the framework
of the realistic {\em complete electrode model} \cite{Hanke11c}.

Our first result
shows that the bisweep data carry the same information as the whole
(relative) Neumann-to-Dirichlet map; the proof is based on a simple
polarization identity.

\begin{theorem}
\label{thm:polarization}
Let the above assumptions on $D$ and $\sigma$ hold.
Then, the bisweep data $\varsigma_\sigma: \partial D \times \partial D \to \R$
determine the whole Neumann-to-Dirichlet map $\Lambda_\sigma$.
\end{theorem}

\begin{proof}
As the considered Neumann-to-Dirichlet maps are self-adjoint, we have
\begin{align*}
2 \, [ (\Lambda_\sigma & - \Lambda_\bg) (\delta_x - \delta_y) ] (z)
-  2\, [ (\Lambda_\sigma - \Lambda_\bg) (\delta_x - \delta_y) ] (x) \\[1mm]
& =   \langle (\delta_z - \delta_x), (\Lambda_\sigma - \Lambda_\bg)  
(\delta_x - \delta_y) \rangle_{\partial D} + \langle (\delta_x - \delta_y), (\Lambda_\sigma - \Lambda_\bg)  
(\delta_z - \delta_x) \rangle_{\partial D} \\[1mm]
& =  \langle (\delta_z - \delta_x), (\Lambda_\sigma - \Lambda_\bg) (\delta_z - \delta_y) \rangle_{\partial D}
+ \langle (\delta_x - \delta_y), (\Lambda_\sigma - \Lambda_\bg)  
(\delta_z - \delta_y) \rangle_{\partial D} \\[1mm]
& \ \ + \langle (\delta_z - \delta_x), (\Lambda_\sigma - \Lambda_\bg)  
(\delta_x - \delta_z) \rangle_{\partial D} 
+\langle (\delta_x - \delta_y), (\Lambda_\sigma - \Lambda_\bg)  
(\delta_y - \delta_x) \rangle_{\partial D} \\[1mm]
& = \varsigma_\sigma(z,y) - \varsigma_\sigma(z,x) - \varsigma_\sigma(x,y),
\end{align*}
for any $x,y,z \in \partial D$. Now, we fix $x, y$ and let $z$ vary 
over all location on $\partial D$, which means that the left-hand side
of the above chain of equalities samples one representative in the quotient
equivalence class $2 (\Lambda_\sigma  - \Lambda_\bg) (\delta_x - \delta_y)$,
namely the one with grounding at $x$.
Because for each $z$ the corresponding right-hand side is a linear
combination of three bisweep data, we deduce that the knowledge of 
$\varsigma_\sigma: \partial D \times \partial D \to \R$ determines 
$(\Lambda_\sigma  - \Lambda_\bg) (\delta_x - \delta_y)$ for any 
$x,y \in \partial D$. On the other hand, the linear span of the set
\begin{equation}
\label{diff_deltat}
\{ \delta_x - \delta_y \}_{x,y \in \partial D}
\end{equation}
is dense in $H^s_\diamond (\partial D)$ for small enough $s = s_n \in \R$
due to, say, the denseness of $\mathscr{D}_\diamond(\partial D)$ in $H^s_\diamond (\partial D)$, a suitable quadrature rule on $\partial D$, and the Sobolev 
embedding theorem for the dual $H^{-s}(\partial D) / \C$. (In fact,
the linear span of \eqref{diff_deltat} remains dense if $y \in \partial D$ is fixed).
This completes the proof.
\end{proof}

Together with the simple polarization argument in the proof of Theorem \ref{thm:polarization},
our main theoretical tool is the fact that in two dimensions 
$\varsigma_\sigma$ extends as
a complex analytic function of two variables to some (interior) 
neighborhood of $\partial D \times \partial D$, 
which will be proven in 
Sections \ref{sec:extension} and \ref{sec:extension2}. 
This result
leads to the following local characterization of the Neumann-to-Dirichlet
map; see Section \ref{sec:uniqueness} for the proof. Here and in what follows, 
we denote by
$$
\mathfrak{D}_z  =  \{ f \in \mathscr{D}_\diamond'(\partial D) \ : 
\ {\rm supp}\, f = z \in \partial D \}
$$
the subspace of mean-free distributions that are supported at some fixed 
$z \in \partial D$.

\begin{theorem}
\label{thm:main}
Let the above assumptions on $D$ and $\sigma$ hold, and suppose furthermore
that $n=2$. 
For any fixed $z \in \partial D$, the point measurements of the type
\begin{equation}
\label{point}
\langle f,
(\Lambda_\sigma - \Lambda_\bg) f \rangle_{\partial D}, 
\qquad f \in \mathfrak{D}_z
\end{equation}
determine the whole Neumann-to-Dirichlet map $\Lambda_\sigma$.
\end{theorem}

Since the knowledge of the Neumann-to-Dirichlet map is equivalent to
that of the Dirichlet-to-Neumann map, the works of Astala and P\"aiv\"arinta
\cite[Theorem 1]{Astala06} and Astala, Lassas and P\"aiv\"arinta \cite[Theorem 1]{Astala05}
provide us immediately with\label{aniso:calderon}
the following uniqueness result for the Calder\'on problem.

\begin{corollary}
\label{cor:main}
Let the assumptions of Theorem \ref{thm:main} hold. Then, the point 
measurements~\eqref{point} for any fixed $z \in \partial D$
 uniquely define an isotropic conductivity $\sigma$. If the conductivity
 $\sigma$ is anisotropic, the point measurements~\eqref{point} determine 
$\sigma$ uniquely up to a pushforward by a boundary-fixing 
$H^1$-diffeomorphism of $D$ onto itself.
\end{corollary}

Compared to previous local uniqueness results for the Calder\'on problem,
Corollary \ref{cor:main} is truly local: the applied
current patterns are the distributions supported at a single point and 
the resulting relative boundary potentials are sampled by these 
same distributions; cf.~\cite[Corollary 1.1 and Theorem 1.2]{Imanuvilov10}.
On the negative side, the assumption that $D$ is two-dimensional and
$\sigma$ equals $1$ close to $\partial D$ seems to be inherent in the proof of 
Theorem~\ref{thm:main}.

\begin{remark}
\label{countably:infinite}
  Suppose $\Gamma \times \Gamma \subset \partial D \times \partial
  D$ is a countably infinite set with the accumulation point $(z,z)$.
  The restriction of the bisweep data $\varsigma_\sigma$ to $\Gamma \times
  \Gamma$ determines the point measurements~\eqref{point} and
  therefore, Theorem~\ref{thm:main} and Corollary~\ref{cor:main} become
  applicable. Indeed, 
  it follows straightforwardly from the polarization 
  identity in the proof of Theorem~\ref{thm:polarization} that the assumed 
  measurements define $\langle f, (\Lambda_\sigma -
  \Lambda_{\bf 1}) g \rangle_{\partial D}$ for all $f, g \in \mathfrak U_\Gamma$
  where 
$$
\mathfrak{U}_\Gamma = {{\rm span}} \{ \delta_x \}_{x \in \Gamma} \cap 
\mathscr{D}_\diamond'(\partial D).
$$
  Since $\mathfrak{D}_z$ belongs to the closure of $\mathfrak{U}_\Gamma$ in the 
(weak) topology of $\mathscr{D}'(\partial D)$,
  the point measurements~\eqref{point} are determined.
\end{remark}

\section{Proof of the main results}
\label{sec:proof}

\subsection{A factorization of the N-to-D map}
\label{sec:factor}

Choose $\Omega \subset \R^n$ to be a simply connected 
$\mathscr{C}^{\infty}$-domain, such that $\Sigma \subset \Omega$ and 
$\overline{\Omega} \subset D$. We define an auxiliary operator 
$$
A : f \mapsto \frac{\partial u_\bg}{\partial \nu}\big|_{\partial \Omega},
\quad \mathscr D'_\diamond(\partial D) \rightarrow \mathscr D_\diamond(\partial \Omega),
$$
where $u_\bg$ is the background solution corresponding to the boundary current pattern~$f$. Notice that $u_{\bg}$ is well defined for any  
$f \in \mathscr D'_\diamond(\partial D)$ because each such (compactly supported)
current density
belongs to $H^s_\diamond(\partial D)$ for some $s=s_f \in \R$, and 
consequently $A$ is also well defined due to the Gauss
divergence theorem and interior
elliptic regularity  (cf., e.g., \cite{Lions72}). Moreover, $A$ is bounded as 
a map from $H^s_{\diamond}(\partial D)$ to $H^{-s}_{\diamond}(\partial \Omega)$
for any $s \in \R$; see, e.g., \cite{Lions72}. The following factorization
can be considered a variant of \cite[Theorem 3.1]{Hanke11b}.

\begin{theorem}\label{thm:factorization}
The operator $\Lambda_\sigma - \Lambda_\bg : \mathscr D'_\diamond(\partial D) \rightarrow \mathscr D(\partial D)/\C$ can be factorized as
\begin{equation}
\label{factorization}
	\Lambda_\sigma - \Lambda_\bg = A' G A
\end{equation}
where $G : \mathscr D'_\diamond(\partial \Omega) \rightarrow \mathscr D(\partial \Omega)/\C$ can be interpreted as a bounded map from $H^{s}_\diamond(\partial \Omega)$ to $H^{-s}(\partial \Omega)/\C$ for any $s \in \R$.
\end{theorem}

\begin{proof}
According to \cite[Corollary 3.2]{Hanke11a}, the Neumann-to-Dirichlet  
map can be factorized as
\[
	\Lambda_\sigma - \Lambda_\bg = B' F B,
\]
where $B : \mathscr D'_\diamond(\partial D) \rightarrow \mathscr D(\partial \Omega)/\C$ maps $f$ to $u_\bg|_{\partial \Omega}$ and $F : H^{-s+1}(\partial \Omega)/\C \rightarrow H^{s-1}_\diamond(\partial \Omega)$ is a bounded operator for any
$s \in \R$.
Clearly, one can write $B = \lambda A$, where $\lambda : H^{-s}_\diamond(\partial \Omega) \rightarrow H^{-s+1}(\partial \Omega)/\C$, the Neumann-to-Dirichlet map for the Laplacian in $\Omega$, is bounded (cf.~\eqref{neumann:dirichlet:Hs:bounded}).
Therefore,
\[
	\Lambda_\sigma - \Lambda_\bg = A' \lambda' F \lambda A =: A' G A,
\]
where the bounded dual operator $\lambda' : H^{s-1}_\diamond(\partial \Omega) \rightarrow H^{s}(\partial \Omega)/\C$ is, in fact, identical to $\lambda$, but interpreted as an operator between different Sobolev spaces. 
Consequently, $G = \lambda' F \lambda : H^{-s}_\diamond(\partial \Omega) \rightarrow H^{s}(\partial \Omega)/\C$ is bounded for any $s \in \R$, which completes the proof.
\end{proof}

In what follows, we  interpret $G$ to be a bounded 
operator from $L^2(\partial \Omega)$ to itself by identifying
it with
$$
P' G P: L^2(\partial \Omega) \to L^2(\partial \Omega)
$$
where $P: L^2(\partial \Omega) \to L^2_\diamond(\partial \Omega)$ is
an orthogonal projection and $P': L^2(\partial \Omega)/ \C \to 
L^2(\partial \Omega)$ is its dual. It is easy to check that $P'$ 
picks the unique mean-free element of an equivalence
class in $L^2(\partial \Omega)/ \C$. We continue to denote this newly 
defined $G$ by the original symbol, and note that the factorization 
\eqref{factorization} remains valid because the range of $A$ consists
of mean-free elements and $A'$ does not `see' the 
constant function.

In particular, take note that Theorem~\ref{thm:factorization} provides
the presentation
\begin{equation}
\label{bisweep2}
\varsigma_\sigma(x,y) = \langle A(\delta_x - \delta_y), GA (\delta_x - \delta_y)
\rangle_{\partial \Omega}, \qquad x,y \in \partial D
\end{equation}
for the bisweep data defined originally by \eqref{bisweep}.   

\subsection{Holomorphic extension of bisweep data in the
unit disk}
\label{sec:extension}

In this section, we assume that $D = B \subset \R^2$ is the open unit disk and 
note that 
the gradient of the corresponding background solution for the Laplacian
$u_\bg^{z_1,z_2}$, $z_1,z_2 \in \partial D$, with the boundary current 
density  $f = \delta_{z_2} - \delta_{z_1}$ is  (cf., e.g., \cite{Hanke11a})
\begin{equation}
\label{eq:nabla:u0:real}
\begin{split}
	\nabla u_\bg^{z_1,z_2}(x)
	&= \frac1{\pi}\left( \frac{x-z_1}{|x-z_1|^2} - \frac{x-z_2}{|x-z_2|^2} \right), \qquad x \in D.
\end{split}
\end{equation}
We identify the mapping $(x,z_1,z_2) \mapsto \nabla u_{\bg}^{z_1,z_2}(x)$ (from $D \times \partial D^2 \subset \R^2 \times \R^2 \times \R^2$ to $\R^2$) with a complex function $(\xi,\zeta_1,\zeta_2) \mapsto v(\xi,\zeta_1,\zeta_2)$, which is a map from $D \times \partial D^2 \subset \C \times \C^2$ to $\C$ 
(cf.~\cite{Hanke11b,Hanke11a,Hyvonen11}). To be more precise,
\begin{equation}
\label{eq:sweep:data:v}
	v(\xi,\zeta_1,\zeta_2)
= \frac1{\pi}\left( \tfrac{\xi-\zeta_1}{(\xi-\zeta_1)\overline{(\xi-\zeta_1)}} - 
\tfrac{\xi-\zeta_2}{(\xi-\zeta_2)\overline{(\xi-\zeta_2)}} \right)
= \frac1{\pi}\left(  \frac{\zeta_1}{\zeta_1 \overline \xi- 1} - \frac{\zeta_2}{\zeta_2 \overline \xi- 1} \right),
\end{equation}
where we took advantage of the fact that $|\zeta_1|^2 = |\zeta_2|^2 = 1$. For fixed 
$\xi \in D$, $v(\xi, \cdot,\cdot)$ extends as a holomorphic function to the set $(\C \setminus \{ 1/\overline{\xi} \}) \times (\C \setminus \{ 1/\overline{\xi} \}) = (\C \setminus \{ 1/\overline{\xi} \})^2 \subset \C^2$; let us denote this extension by $w_1(\xi, \cdot, \cdot)$.
Analogously, the complex conjugate of $v$, i.e.,
\[
	\overline{v(\xi,\zeta_1,\zeta_2)} = \frac1{\pi}\left( \frac1{\xi - \zeta_1} - \frac1{\xi- \zeta_2} \right),
\]
can be extended  as a holomorphic function, 
say $w_2(\xi, \cdot, \cdot)$, to $(\C \setminus \{\xi\})^2$. It thus
follows that also the real and imaginary parts of $v(\xi, \cdot,\cdot)$ have holomorphic 
extensions
\begin{equation}
\label{holo_real_imag}
\begin{split}
v_1(\xi, \zeta_1, \zeta_2) &= \frac{1}{2}\big(w_1(\xi,\zeta_1, \zeta_2) + w_2(\xi, \zeta_1, \zeta_2)\big), \\ 
v_2(\xi, \zeta_1, \zeta_2) &= \frac{1}{2 i}\big(w_1(\xi,\zeta_1, \zeta_2) - w_2(\xi, \zeta_1, \zeta_2)\big)
\end{split}
\end{equation}
to $(\C \setminus (\{ \xi \} \cup \{1/\overline{\xi}\}))^2$. We denote 
$V(\xi,\zeta_1,\zeta_2)= [v_1(\xi,\zeta_1,\zeta_2), v_2(\xi,\zeta_1,\zeta_2)]^{\rm T}$.

Let $U \subset \C$ be an open neighbourhood of $\partial D$ such that $\overline \Omega \cap \overline U = \emptyset$ and $\overline{\Omega^*} \cap \overline U = \emptyset$, where $\Omega^*$ is the reflection of $\Omega$ with respect to
the unit circle $\partial D$.
Due to \eqref{bisweep2}, the bisweep data can be identified with the restriction $\varsigma_\sigma(\zeta_1,\zeta_2)|_{(\partial D)^2}$ of the function $\varsigma_\sigma(\zeta_1,\zeta_2) : U^2 \rightarrow \C$,
\begin{equation}
\begin{split}
\label{extension}
	\varsigma_\sigma(\zeta_1,\zeta_2) = 
	\langle A(\zeta_1,\zeta_2), G A(\zeta_1,\zeta_2) \rangle_{\partial \Omega} 
	= \int_{\partial \Omega} h(\xi,\zeta_1,\zeta_2) (G_y h(y,\zeta_1,\zeta_2))(\xi) \d s_\xi,
\end{split}
\end{equation}
where $\d s_\xi$ corresponds to the (real) arc length measure on 
$\partial \Omega$ and $A : U^2 \rightarrow L^2(\partial \Omega)$ is defined by
\begin{equation*}
	(A(\zeta_1,\zeta_2))(\xi) = h(\xi,\zeta_1,\zeta_2) := \nu_\xi \cdot V(\xi,\zeta_1,\zeta_2),
\end{equation*} 
with $\nu_\xi$ being the (real) unit normal of $\partial \Omega$ at $\xi$. 
We will now show that $\varsigma_\sigma$ is holomorphic in $U^2$.

\begin{lemma}\label{lemma:A:holomorphic}
The operator $A(\zeta_1,\zeta_2)$ is holomorphic in $\zeta_1 \in U$ (resp. $\zeta_2 \in U$) for an arbitrary fixed value of $\zeta_2 \in U$ (resp. $\zeta_1 \in U$).
\end{lemma}

\begin{proof}
Let $M > 0$ be a real constant that satisfies
 \[
      \bigg| \frac{\partial^2}{\partial \zeta_1^{2}}  v_l(\xi,\zeta_1,\zeta_2) \bigg| \leq M
 \]
 for all $\zeta_1, \zeta_2 \in U$, $\xi \in \partial \Omega$ and $l=1,2$.
Fix arbitrary $\zeta_1 \in U$ and let $r > 0$ be such that 
$\{ w \in \C \, : \, | w - \zeta_1 | < r\}\subset U$. By representing
the difference $v_l(\xi,\zeta_1 + \eta,\zeta_2) - v_l(\xi,\zeta_1,\zeta_2)$
as a complex line integral and subsequently applying the same idea
to the derivative of $v_l$ with respect to $\zeta_1$,
it follows easily that
 \begin{align*}
 &\left|\frac{v_l(\xi,\zeta_1 + \eta,\zeta_2) - 
 v_l(\xi,\zeta_1,\zeta_2)}{\eta} - \frac{\partial v_l}{\partial 
 \zeta_1}(\xi,\zeta_1,\zeta_2)\right|
      \leq \frac{1}{2}M |\eta|, \qquad l=1,2,
 \end{align*}
for all $\xi \in \partial \Omega$, $\zeta_2 \in U$ and 
$0 \not = \eta \in \C$ such that $| \eta | < r$. In consequence,
 \begin{align*}
      \left\| \frac{ h(\cdot,\zeta_1 + \eta,\zeta_2) - 
 h(\cdot,\zeta_1,\zeta_2)}{\eta} - \frac{\partial h}{\partial 
 \zeta_1}(\cdot,\zeta_1,\zeta_2) \right\|_{L^2(\partial \Omega)}
      \leq  M |\eta| \sqrt{|\partial \Omega|},
 \end{align*}
which means that $A$ is holomorphic in $\zeta_1$. 
The same argument can  be applied to~$\zeta_2$, and the claim follows.
\end{proof}

\begin{lemma}\label{lemma:bilinear:derivative}
Let $X$ and $Y$ be complex Banach spaces, $\la \cdot, \cdot \ra : X \times Y \rightarrow \C$ a bounded bilinear form, and $f : U \rightarrow X$, $g : U \rightarrow Y$ differentiable in an open set $U \subset \C$ (resp. $U \subset \R$). Then,
\begin{equation*}
	\frac{\d}{\d z} \la f(z), g(z) \ra = \la f'(z), g(z) \ra + \la f(z), g'(z) \ra
\end{equation*}
for any $z \in U$. In particular, if $f$ and $g$ are holomorphic, then
so is the map $U \ni z \mapsto \la f(z), g(z) \ra \in \C$.
\end{lemma}

\begin{proof}
The assertion follows from essentially the same argument as 
the standard product rule of calculus.
\end{proof}

\begin{theorem}
The bisweep data extend to a holomorphic function $\varsigma_\sigma : U^2 \rightarrow \C$, where $U$ is an open neighborhood of $\partial D \subset \C$.
\end{theorem}

\begin{proof}
We have interpreted $G$ as a bounded operator from $L^2(\partial \Omega)$ to 
itself, which makes
\[
	(p,q) \mapsto \langle p, G q \rangle_{\partial \Omega} = \int_{\partial \Omega} p(\xi) (Gq)(\xi) \d s_\xi
\]
a bounded bilinear form on $L^2(\partial \Omega) \times L^2(\partial \Omega)$. By 
\eqref{extension} and Lemmas \ref{lemma:A:holomorphic} and \ref{lemma:bilinear:derivative}, the extension $\varsigma_\sigma : U^2 \rightarrow \C$ is holomorphic in either variable if the other has an arbitrary fixed value. Due to the Hartogs' theorem 
\cite[Theorem~2.2.8]{Hormander73}, this means that $\varsigma_\sigma$ is, in fact, analytic in $U^2$, that is, it locally coincides with its multi-dimensional complex Taylor series.
\end{proof}

\begin{corollary}
The angular bisweep data $\tilde \varsigma_\sigma : \R^2 \rightarrow \R$,
\begin{equation}
\label{eq:def:angular:bisweep}
	\tilde \varsigma_\sigma(\theta_1,\theta_2) := \varsigma_\sigma(e^{i\theta_1},e^{i\theta_2}),
\end{equation}
is an analytic function.
\end{corollary}

\begin{proof}
Clearly, the definition \eqref{eq:def:angular:bisweep} can be extended to some open set $V^2 \subset \C^2$ such that $\R \subset V$ and $e^{i V} \subset U$, and by the chain rule, it is a holomorphic function of two variables. As the restriction of such a function to $\R^2$, the angular sweep data is analytic, that is, locally coincides with its two-dimensional real Taylor series.
\end{proof}

In what follows, we denote by $D_j$ the derivative with respect to the 
$j$th variable. Due to the theory of analytic continuation, we have

\begin{corollary}\label{corollary:bisweep:derivatives}
The (angular) bisweep data is completely determined by the set of its derivatives
\begin{equation*}
	\{ D_{1}^j D_{2}^k \tilde \varsigma_\sigma(\theta_1,\theta_2) \;:\; j,k \in \mathbb \N_0 \}
\end{equation*}
at an arbitrary point $(\theta_1, \theta_2) \in \R^2$.
\end{corollary}

\subsection{Generalization to smooth domains}
\label{sec:extension2}
Let us then adopt the setting of Theorem~\ref{thm:main} and assume, in particular, that 
$D \subset \R^2$ is a simply connected and bounded domain with a 
$\mathscr{C}^\infty$-boundary.

Let $\Phi$ be a conformal map of the open unit disk $B$ onto $D$. According to
\cite[Section 3.3]{Pommerenke92}, $\Phi|_{\partial B}$ in turn defines
a smooth diffeomorphism of $\partial B$ onto $\partial D$. 
In the spirit of  \eqref{eq:def:angular:bisweep}, we introduce the 
angular bisweep data $\tilde \varsigma_\sigma : \R^2 \rightarrow \R$,
\begin{equation}
\label{eq:def:angular:bisweep2}
	\tilde \varsigma_\sigma(\theta_1,\theta_2) := \varsigma_\sigma\big(\Phi(e^{i\theta_1}),\Phi(e^{i\theta_2})\big),
\end{equation}
and generalize Corollary \ref{corollary:bisweep:derivatives} to our new
framework.

\begin{corollary}\label{corollary:bisweep:derivatives2}
Assume that $D$ satisfies the above assumptions.
Then, the angular bisweep data \eqref{eq:def:angular:bisweep2}
is completely determined by the set of its derivatives
\begin{equation}
\label{eq:bisweep:angular:derivatives2}
	\{ D_{1}^j D_{2}^k \tilde \varsigma_\sigma(\theta_1,\theta_2) \;:\; j,k \in \mathbb \N_0 \}
\end{equation}
at an arbitrary point $(\theta_1, \theta_2) \in \R^2$.
\end{corollary}

\begin{proof}
It follows directly from the argument in the proof of 
\cite[Theorem 3.2]{Hakula11} that
\begin{equation}
\label{conformal_sweep}
\varsigma_{\sigma^*} := 
\varsigma_\sigma(\Phi|_{\partial B}(\cdot), \Phi|_{\partial B}(\cdot)) : 
\partial B \times \partial B \to \R
\end{equation}
is the bisweep data corresponding to the unit disk and the pull-back
conductivity 
$$
\sigma^*:= J_\Phi^{-1} (\sigma \circ \Phi) (J_\Phi^{-1})^{\rm T} \det J_\Phi,
$$
where $J_\Phi$ is the Jacobian matrix of $\Phi$ (interpreted as a map 
from $\R^2$ to itself).
Notice that $\sigma^*$ is a feasible conductivity, i.e., it is strictly
positive definite and equals one in some interior neighborhood of $\partial B$,
due to the basic properties of conformal mappings (cf.~\cite[Section~3]{Hakula11}).

By definition, 
$$
\tilde{\varsigma}_{\sigma^*} = \tilde{\varsigma}_\sigma
$$
where $\tilde{\varsigma}_{\sigma^*}$ is the angular sweep data for $\sigma^*$ 
on $\partial B$ defined by \eqref{eq:def:angular:bisweep} and  
$\tilde{\varsigma}_\sigma$ is given by~\eqref{eq:def:angular:bisweep2}.
According to Corollary~\ref{corollary:bisweep:derivatives}, 
$\tilde{\varsigma}_{\sigma^*}$ is determined by the set of its derivatives~\eqref{eq:bisweep:angular:derivatives2}, which completes the proof.
\end{proof}

\begin{remark}
\label{remark2}
  The smoothness of $\partial D$ is needed for Corollary~\ref{corollary:bisweep:derivatives2} so that the definition \eqref{bisweep} makes sense --- for which
less regularity would certainly suffice --- and that the assumptions of 
\cite[Theorem~3.2]{Hakula11} are satisfied. In consequence, 
  if \cite[Theorem~3.2]{Hakula11} extends to more
  general domains with less regular boundaries (as it does), 
  Corollary~\ref{corollary:bisweep:derivatives2} adopts 
the corresponding regularity  assumptions on $\partial D$. 
 This same reduction of smoothness carries over
to Theorem~\ref{thm:main} and Corollary~\ref{cor:main} with the extra
requirement of \emph{local} $\mathscr{C}^\infty$-smoothness around $z\in \partial D$, 
as apparent from the proof presented in the following section.
\end{remark}

\subsection{Uniqueness by point measurements}
\label{sec:uniqueness}
Let us adopt the assumptions and the notation of Section \ref{sec:extension2}. 
We will prove the claim of Theorem \ref{thm:main} by utilizing the 
pointwise-supported mean-free distributions $\{ \delta_\theta^{(j)} \}_{j=1}^\infty \subset
\mathfrak{D}_{\Phi(e^{i \theta})}$ defined via
$$
\langle \delta_\theta^{(j)}, \varphi \rangle_{\partial D} = 
\frac{{\rm d}^j}{{\rm d} \vartheta^j} \varphi(\Phi(e^{i \vartheta}))|_{\vartheta=\theta},
\qquad \varphi \in \mathscr{D}(\partial D), \ \ j \in \N.
$$ 
We extend this definition in the natural way to the case $j=0$
by requiring that $\langle \delta_\theta^{(0)}, \varphi 
\rangle_{\partial D} = \langle \delta_\theta, \varphi 
\rangle_{\partial D} = \varphi(\Phi(e^{i \theta}))$, but note that this
standard (angular) delta distribution is not mean-free. 
It follows easily from the Sobolev embedding theorem that the mapping $\theta \mapsto \delta^{(k)}_{\theta}$ is differentiable, say, from $\R$ to $H^{-k-3}(\partial D)$, $k \in \N_0$, and that the corresponding derivative  is $\theta \mapsto \delta^{(k+1)}_\theta$. 

The angular bisweep data of \eqref{eq:def:angular:bisweep2}
allows the representation
$$
\tilde{\varsigma}_{\sigma}(\theta_1,\theta_2) = \left\langle (\delta_{\theta_1} - \delta_{\theta_2}), (\Lambda_\sigma - \Lambda_\bg) (\delta_{\theta_1} - \delta_{\theta_2}) \right \rangle_{\partial D}
$$
for any $\theta_1, \theta_2 \in \R$.
Obviously,
$$
\tilde{\varsigma}_{\sigma}(\theta,\theta) = 0
$$
for any $\theta \in \R$.
Due to Lemma~\ref{lemma:bilinear:derivative} and the boundedness (and self-adjointness) of 
$\Lambda_{\sigma} - \Lambda_\bg: H^{s}_\diamond(\partial D) \to H^{-s}(\partial D)/\C$ for any $s \in \R$, we have
\begin{equation}
\label{eq:bisweep:first:derivative}
D_1 \tilde{\varsigma}_{\sigma}(\theta_1,\theta_2)
= 2 \la \delta^{(1)}_{\theta_1}, (\Lambda_{\sigma} - \Lambda_\bg) 
(\delta_{\theta_1} - \delta_{\theta_2}) \ra,
\end{equation}
which vanishes for $\theta_1 = \theta_2 = \theta$. Similarly, 
the $k$th derivative of the angular bisweep data with respect to the first variable at $(\theta_1,\theta_2)$ reads
\begin{equation}
\label{eq:bisweep:kth:derivative}
	D_1^k \tilde{\varsigma}_{\sigma}(\theta_1,\theta_2)
	= 2 \la \delta^{(k)}_{\theta_1}, (\Lambda_{\sigma} - \Lambda_\bg) 
(\delta_{\theta_1} - \delta_{\theta_2}) \ra +
	\sum_{j=1}^{k-1} {k \choose j} \la  \delta^{(j)}_{\theta_1}, (\Lambda_{\sigma} - \Lambda_\bg)  \delta^{(k-j)}_{\theta_1} \ra,
\end{equation} 
meaning that
\[
	D_1^k \tilde{\varsigma}_{\sigma}(\theta,\theta)
	= \sum_{j=1}^{k-1} {k \choose j} \la \delta^{(j)}_{\theta}, (\Lambda_{\sigma} - \Lambda_\bg) \delta^{(k-j)}_{\theta} \ra, \qquad \theta \in \R, \ k \geq 2.
\]
It clearly holds that $D_2^k \tilde{\varsigma}_{\sigma}(\theta,\theta) = D_1^k \tilde{\varsigma}_{\sigma}(\theta,\theta)$ for any $\theta \in \R$ and 
$k \in \N$.
Moreover, taking the $l$th derivative of \eqref{eq:bisweep:first:derivative} and \eqref{eq:bisweep:kth:derivative} with respect to the second variable results in
\begin{equation*}
	D_1^k D_2^l  \tilde{\varsigma}_{\sigma}(\theta_1,\theta_2)
	= -2 \la \delta^{(k)}_{\theta_1}, (\Lambda_{\sigma} - \Lambda_\bg) \delta^{(l)}_{\theta_2} \ra, \qquad \theta_1, \theta_2 \in \R, \  k, l \in \N.
\end{equation*}
In consequence, we have altogether 
shown that any partial derivative of the angular 
bisweep data at $\theta_1= \theta_2 = \theta \in \R$ is either known 
to vanish or can be given as a linear combination of terms of the form
$$
\langle f,
(\Lambda_\sigma - \Lambda_\bg) g \rangle_{\partial D}, 
\qquad f, g \in \mathfrak{D}_z,
$$
where $z = \Phi(e^{i \theta})$ can be chosen arbitrarily via the choice
of $\theta$. Due to the standard polarization identity 
$$
4 \langle f,
(\Lambda_\sigma - \Lambda_\bg) g \rangle_{\partial D} = 
\langle (f+g),
(\Lambda_\sigma - \Lambda_\bg) (f+g) \rangle_{\partial D} - \langle (f-g),
(\Lambda_\sigma - \Lambda_\bg) (f-g) \rangle_{\partial D},
$$ 
this means that the knowledge
of the point measurements \eqref{point} for all $f \in \mathfrak{D}_z$ 
(with fixed $z \in \partial D$) implies
the knowledge of all the derivatives 
\eqref{eq:bisweep:angular:derivatives2} at the corresponding
polar angle $\theta_1 = \theta_2 = \theta$.

The statement of Theorem~\ref{thm:main} follows now from 
Corollary~\ref{corollary:bisweep:derivatives2} and 
Theorem~\ref{thm:polarization}.

\section*{Acknowledgements}
Nuutti Hyv\"onen and Petteri Piiroinen would like to thank John Sylvester,
Bastian Harrach, Lauri Oksanen, Roland Griesmaier and Martin Simon for
the discussions about two-electrode measurements at the Oberwolfach
workshop on Inverse Problems for Partial Differential Equations (ID: 1208b)
organized by Martin Hanke, Andreas Kirsch, William Rundell and Matti Lassas.

\end{document}